\documentclass[12pt]{article}
\usepackage[utf8]{inputenc}
\usepackage{amsmath,amssymb,amsthm}
\usepackage{geometry}
\theoremstyle{remark}
\theoremstyle{definition}
\newtheorem{theorem}{Theorem}[section]
\newtheorem{lemma}[theorem]{Lemma}

\newtheorem{definition}[theorem]{Definition}
\newtheorem{remark}[theorem]{Remark}
\newtheorem{proposition}[theorem]{Proposition}

\newcommand{\iu}{\mathrm{i}}
\usepackage{algorithm}
\usepackage{algpseudocode}
\usepackage[numbers,sort&compress]{natbib}
\title{On spanning trees whose degrees are congruent to one modulo $\ell$}
\author{Zhidan Yan \quad Wei Wang\thanks{Corresponding author\\
		E-mail addresses: wangwei.math@gmail.com (W. Wang); yanzhidan.math@gmail.com (Z. Yan)}\\[2ex]
		{\footnotesize School of Mathematics, Physics and Finance, Anhui Polytechnic University, Wuhu 241000, P. R. China}}
\date{}

\begin{document}
	\maketitle
	
\begin{abstract}
	An $\ell$-congruent spanning tree of a nontrivial connected graph is
	a spanning tree in which every vertex has degree congruent to one
	modulo $\ell$. This notion provides a common generalization of
	classical spanning trees and odd spanning trees. We show, via a constructive greedy algorithm, that every $n$-vertex graph $G$ satisfying $n\equiv2\pmod{\ell}$ and $\delta(G)>\frac{(\ell-1)n}{\ell}$ has an $\ell$-congruent spanning tree. For the special case of odd spanning trees ($\ell=2$), our algorithmic approach simplifies the original proof by Zheng and Wu.  We also derive formulas for the
	numbers of $\ell$-congruent spanning trees in complete graphs and
	complete bipartite graphs. These formulas specialize to the classical
	spanning-tree formulas when $\ell=1$ and to the corresponding
	odd-spanning-tree formulas when $\ell=2$.
\end{abstract}

\vspace{1em}


	
	\section{Introduction}
	A spanning tree of a graph $G$ is a connected spanning subgraph
	containing no cycle. A graph has a spanning tree if and only if it is
	connected. Let $\tau(G)$ denote the number of spanning trees of $G$. Counting spanning trees is a classical problem in graph theory and is
	closely connected with several areas of mathematics, statistical
	physics, and theoretical computer science. Recent work has addressed spanning-tree enumeration for line graphs
	\cite{dong2017,gong2018}, complete bipartite and multipartite graphs
	with prescribed spanning forests \cite{dong2022,li2023,li2025},
	and other structured or weighted graph families
	\cite{gong2024,yang2025,zhou2021}.

	The foundation for these studies was laid by the pioneering work on basic graph classes. In particular, the following two fundamental formulas for the number of spanning trees in complete graphs and complete bipartite graphs are well known.
	\begin{theorem}[\cite{cayley1889}]
		$\tau(K_n)=n^{n-2}$.
	\end{theorem}
	\begin{theorem}[\cite{fiedler1958}]
		$\tau(K_{m,n})=m^{n-1}n^{m-1}$.
	\end{theorem}
Recently, Zheng and Wu \cite{zheng2025} introduced the notion of an
\emph{odd spanning tree}, namely, a spanning tree in which every vertex
has odd degree. By the handshaking lemma, a graph of odd order cannot
have an odd spanning tree. For graphs of even order, they obtained a
sufficient condition for the existence of an odd spanning tree.
	\begin{theorem}[\cite{zheng2025}]\label{md}
		Let $n$ be a positive even number. If $G$ is a connected graph of order $n$ with $\delta(G)\ge \frac{n}{2}+1$, then  $G$ has an odd spanning tree.
	\end{theorem}
Let $\tau_o(G)$ denote the number of odd spanning trees of $G$.
 It is natural to determine
$\tau_o(G)$ for important families of graphs. Feng, Chen, and Wu
\cite{feng2026} obtained a formula for $\tau_o(K_n)$. More recently,
the corresponding formula for $K_{m,n}$ was established by Ge and Yu
\cite{ge2026} and, independently, by Xu and Xu \cite{xu2026}.
	\begin{theorem}[\cite{feng2026}]\label{tc}
		For the complete graph $K_n$, we have
		\begin{equation*}
			\tau_o(K_n)=\frac{1}{2^n}\sum_{k=0}^n\binom{n}{k}(2k-n)^{n-2}.
		\end{equation*}
	\end{theorem}
	\begin{theorem}[\cite{ge2026,xu2026}]
		For the complete bipartite graph $K_{m,n}$, we have	\begin{equation*}\tau_o(K_{m,n})=\frac{1}{2^{m+n}}
		\left[\sum_{k=0}^m\binom{m}{k}(2k-m)^{n-1}
		\right] \left[\sum_{s=0}^n\binom{n}{s}(2s-n)^{m-1}
		\right].
		\end{equation*}
	\end{theorem}
We note that the above form of Theorem~\ref{tc} is due to Ge and Yu
\cite{ge2026}. The original statement in \cite{feng2026} distinguishes
between odd and even values of $n$:
\begin{equation*}
	\tau_o(K_n)=
	\begin{cases}
		0, & \text{if $n$ is odd},\\[2mm]
		\displaystyle
		\frac{1}{2^n}
		\sum_{k=0}^n\binom{n}{k}(2k-n)^{n-2},
		& \text{if $n$ is even}.
	\end{cases}
\end{equation*}
Indeed, in the summation formula of Theorem~\ref{tc}, if $n$ is odd, the $k$-th summand is the negative of the $(n-k)$-th summand, so the total sum automatically vanishes. Similarly,
$\tau_o(K_{m,n})=0$ whenever at least one of $m$ and $n$ is even.
In particular, $\tau_o(K_{m,m})=0$ for even $m$, showing that the
minimum-degree condition in Theorem~\ref{md} is sharp in the sense  the  threshold $\delta(G)\ge n/2+1$ cannot be weakened to  $\delta(G)\ge n/2$

The main aim of this paper is to place classical spanning trees and odd spanning trees in a common framework. To this end, we introduce
the following notion. 
	\begin{definition}
		Let $\ell$ be a positive integer. A nontrivial tree $T$ is called
		an $\ell$-congruent tree if the degrees of all vertices of $T$ are
		congruent modulo $\ell$. Since every nontrivial tree has a leaf,
		this is equivalent to
		\[
		d_T(v)\equiv 1\pmod{\ell}
		\qquad\text{for every }v\in V(T).
		\]
		An $\ell$-congruent spanning tree of a graph $G$ is a spanning tree
		of $G$ that is an $\ell$-congruent tree.
	\end{definition}
\begin{remark}
	We restrict the definition of an $\ell$-congruent tree to nontrivial
	trees. Indeed, the unique vertex of $K_1$ has degree zero. Thus, the
	degrees of its vertices are trivially congruent modulo $\ell$, but
	its degree is not congruent to one modulo $\ell$ unless $\ell=1$.
\end{remark}
	\begin{remark}
	Since all integers are congruent modulo $1$, a $1$-congruent spanning
	tree is simply an ordinary spanning tree. Moreover, a
	$2$-congruent spanning tree is precisely an odd spanning tree.
	Thus, $\ell$-congruent spanning trees provide a common framework for
	classical spanning trees and odd spanning trees.
	\end{remark}
Our first main result gives a necessary condition and a sufficient minimum-degree condition for the existence of an $\ell$-congruent
spanning tree.
	\begin{theorem}\label{ns}
		Let $G$ be a connected graph of order $n\ge 2$. If $G$ has an $\ell$-congruent spanning tree, then $n\equiv 2\pmod{\ell}$. Conversely, if $n\equiv 2\pmod{\ell}$ and $\delta(G)> \frac{(\ell-1)n}{\ell}$, then $G$ has an $\ell$-congruent spanning tree.
	\end{theorem}
Theorem~\ref{ns} generalizes Theorem~\ref{md}. Indeed, when $\ell=2$, the congruence $n\equiv2\pmod{2}$ implies that $n$ is even. Since vertex degrees are integers, the strict inequality $\delta(G)>\frac{n}{2}$ is exactly equivalent to $\delta(G)\ge\frac{n}{2}+1$, precisely recovering the condition of Theorem~\ref{md}. Furthermore, in the special case $\ell=1$, the congruence $n\equiv2\pmod{1}$ is automatic, and the degree condition reduces to $\delta(G)>0$. Hence Theorem~\ref{ns} also recovers the elementary fact that every nontrivial connected graph has a spanning tree.

	Our proof of the sufficient condition is constructive. Starting from
	an arbitrary edge, we repeatedly adjoin $\ell$ new leaves to a vertex
	of the current tree whenever possible. This greedy procedure preserves
	the degree congruence at every step, and the minimum-degree condition
	guarantees that it cannot terminate before producing a spanning tree.
	For $\ell=2$, it yields an alternative algorithmic proof of
	Theorem~
	\ref{md} based on a single uniform extension rule.
	
		\begin{remark}
		While the minimum-degree condition in Theorem~\ref{ns} is sharp for $\ell=2$ (as demonstrated by $K_{m,m}$ with even $m$), We do not know whether this threshold is best possible for $\ell\ge3$. Our constructive proof based on a greedy approach guarantees existence but may overstep the optimal degree threshold. Finding the exact optimal minimum-degree condition for the existence of an $\ell$-congruent spanning tree for $\ell \ge 3$ remains an interesting problem for future research.
	\end{remark}

We also enumerate $\ell$-congruent spanning trees in complete graphs
and complete bipartite graphs. To unify the notation, let $\tau_\ell(G)$ denote the number of $\ell$-congruent spanning trees of $G$, so that $\tau_1(G)=\tau(G)$ and $\tau_2(G)=\tau_o(G)$.	\begin{theorem}\label{kn}
		Let $n\ge 2$. For the complete graph $K_n$, we have
		\begin{equation}
			\tau_{\ell}(K_n) = \frac{1}{\ell^n} \sum_{\substack{k_0+k_1+\cdots+k_{\ell-1}=n \\ k_0,\ldots,k_{\ell-1}\ge 0}} \frac{n!}{k_0! k_1! \cdots k_{\ell-1}!} \left( \sum_{p=0}^{\ell-1} k_p  e^{\frac{2\pi \iu }{\ell}p}
			\right)^{n-2}.
		\end{equation}
	\end{theorem}
	
	\begin{theorem}\label{kmn}
		Let $m,n$ and $\ell$ be positive integers. For the complete bipartite graph $K_{m,n}$, we have
	\begin{equation}
		\begin{aligned}
			\tau_{\ell}(K_{m,n})
			={}& \frac{1}{\ell^{m+n}}
			\left[
			\sum_{\substack{
					k_0+k_1+\cdots+k_{\ell-1}=m\\
				k_0,\ldots,k_{\ell-1}\ge 0
			}}
			\frac{m!}{k_0!k_1!\cdots k_{\ell-1}!}
			\left(
			\sum_{p=0}^{\ell-1}
			k_p e^{\frac{2\pi\iu}{\ell}p}
			\right)^{n-1}
			\right]
			\\
			&\quad\times
			\left[
			\sum_{\substack{
					s_0+s_1+\cdots+s_{\ell-1}=n\\
					s_0,\ldots,s_{\ell-1}\ge 0
			}}
			\frac{n!}{s_0!s_1!\cdots s_{\ell-1}!}
			\left(
			\sum_{q=0}^{\ell-1}
			s_q e^{\frac{2\pi\iu}{\ell}q}
			\right)^{m-1}
			\right].
		\end{aligned}
	\end{equation}
		\end{theorem}
		Theorems~\ref{kn} and~\ref{kmn} simultaneously generalize the
		classical enumeration formulas for spanning trees and the recent
		formulas for odd spanning trees. Indeed, when $\ell=1$, the degree
		condition is vacuous, and Theorems~\ref{kn} and~\ref{kmn} reduce to
		Cayley's formula
		\[
		\tau(K_n)=n^{n-2}
		\]
		and its complete bipartite analogue
		\[
		\tau(K_{m,n})=m^{n-1}n^{m-1},
		\]
		respectively. When $\ell=2$, they reduce to the corresponding formulas
		for odd spanning trees stated above. Thus, both the existence result
		and the enumeration formulas reflect the same unifying principle:
		classical spanning trees and odd spanning trees are the first two
		instances of $\ell$-congruent spanning trees.
		
		The remainder of the paper is organized as follows. In
		Section~\ref{ex}, we prove Theorem~\ref{ns} by means of a
		greedy construction. In Section~\ref{en}, we prove
		Theorems~\ref{kn} and~\ref{kmn} using prescribed-degree enumeration
		formulas and a roots-of-unity filter.
	\section{Existence of  $\ell$-congruent spanning trees}\label{ex}
	In this section, we prove Theorem~\ref{ns}. Its necessary condition
	follows immediately from the handshaking lemma, while its sufficient
	condition is proved by a constructive greedy algorithm.
	
  Let $T$ be an $\ell$-congruent spanning tree of $G$. Using the handshaking lemma for $T$, we obtain
	\begin{equation}\label{hs}
		\sum_{v\in V(T)}d_T(v)=2|E(T)|=2(n-1).
	\end{equation}
Since $d_T(v)\equiv1\pmod{\ell}$ for every $v\in V(T)$, reducing
\eqref{hs} modulo $\ell$ gives
\[
n\equiv2(n-1)\pmod{\ell}.
\]
Therefore,
\[
n\equiv2\pmod{\ell}.
\]
This proves the necessary condition.		
	
To prove the second assertion, we greedily construct an
$\ell$-congruent subtree by repeatedly adjoining $\ell$ new leaves
to a vertex already in the tree. We call such a subtree
\emph{extension-maximal} if every vertex of the subtree has at most
$\ell-1$ neighbors outside it.
	\begin{algorithm}[htbp]
		\caption{Greedy construction of an extension-maximal $\ell$-congruent subtree}
		\label{ag}
		\begin{algorithmic}[1]
			\Require A nontrivial connected graph $G$ and a positive integer $\ell$
			\Ensure An extension-maximal $\ell$-congruent subtree $T$ of $G$
			
			\State Choose an arbitrary edge of $G$, and let $T$ be the tree
			consisting of this edge and its two endvertices.
			
			\While{some vertex $v\in V(T)$ has at least $\ell$ neighbors outside $T$}
			\State Choose any $\ell$ such neighbors of $v$.
			\State Add these $\ell$ vertices and the corresponding $\ell$
			edges incident with $v$ to $T$.
			\EndWhile
			
			\State \Return $T$
		\end{algorithmic}
	\end{algorithm}
	At each iteration, the algorithm adds exactly $\ell$ new vertices.
	Hence it terminates after at most $\lfloor (n-2)/\ell\rfloor$
	iterations.
		\begin{proposition}
		Let $\ell$ be a positive integer, and let $G$ be a connected graph
		of order $n$ such that $n\equiv2\pmod{\ell}$. If
		\[
		\delta(G)>\frac{(\ell-1)n}{\ell},
		\]
		then Algorithm~\ref{ag} returns an $\ell$-congruent spanning tree
		of $G$.
	\end{proposition}
	\begin{proof}
		Let $T$ be the tree returned by Algorithm~\ref{ag}, and write
		\[
		t=|V(T)|,\qquad
		S=V(G)\setminus V(T),\qquad
		s=|S|=n-t.
		\]
		By construction, $T$ is an $\ell$-congruent tree. Indeed, initially
		$T$ consists of a single edge, so both of its vertices have degree
		one. At each iteration, the degree of one existing vertex increases
		by $\ell$, while every new vertex has degree one. Thus all vertex
		degrees remain congruent to one modulo $\ell$.
		
		It remains to prove that $T$ is spanning. If $\ell=1$ and $T$ is
		not spanning, then the connectivity of $G$ guarantees an edge
		between $V(T)$ and $S$. Hence some vertex of $T$ has a neighbor
		outside $T$, so the algorithm cannot have terminated. Therefore
		$T$ is spanning when $\ell=1$.
		
		Assume henceforth that $\ell\ge2$, and suppose, to the contrary,
		that $s>0$. For each $v\in V(T)$, let
		\[
		d_S(v)=|N_G(v)\cap S|
		\]
		denote the number of neighbors of $v$ outside $T$, and set
		\[
		\Delta_S(T)=\max_{v\in V(T)}d_S(v).
		\]
		We derive two lower bounds for $\Delta_S(T)$.
		
		First, every vertex $v\in V(T)$ has at most $t-1$ neighbors in
		$V(T)$. Therefore
		\[
		d_S(v)
		=d_G(v)-|N_G(v)\cap V(T)|
		\ge\delta(G)-(t-1),
		\]
		and hence
		\begin{equation}\label{lowbd1}
			\Delta_S(T)\ge\delta(G)-t+1.
		\end{equation}

Second,	let $E_G(V(T),S)$ denote the set of edges of $G$ with one endpoint
	in $V(T)$ and the other in $S$. Since every vertex of $S$ has at most
	$s-1$ neighbors in $S$, it has at least $\delta(G)-s+1$ neighbors in
	$V(T)$. Hence
	\[
	|E_G(V(T),S)|
	\ge s\bigl(\delta(G)-s+1\bigr).
	\]
	Since
	\[
	|E_G(V(T),S)|
	=\sum_{v\in V(T)}d_S(v),
	\]
	averaging over the $t$ vertices of $T$ gives
		\begin{equation}\label{lowbd2}
			\Delta_S(T)
			\ge
			\left\lceil
			\frac{s\bigl(\delta(G)-s+1\bigr)}{t}
			\right\rceil.
		\end{equation}
		
		On the other hand, the algorithm has terminated, so no vertex of
		$T$ has $\ell$ or more neighbors outside $T$. Thus
		\begin{equation}\label{ub}
			\Delta_S(T)\le\ell-1.
		\end{equation}
		Combining \eqref{lowbd1} with
		\eqref{ub}, we obtain
		\begin{equation}\label{tlb}
			t\ge\delta(G)-\ell+2.
		\end{equation}
		Combining \eqref{lowbd2} with
		\eqref{ub}, and omitting the ceiling, gives
		\begin{equation}\label{st}
			s\bigl(\delta(G)-s+1\bigr)
			\le(\ell-1)t.
		\end{equation}
		
		Since the algorithm starts with two vertices and adds exactly
		$\ell$ vertices at each iteration,
		\[
		t\equiv2\pmod{\ell}.
		\]
		Together with $n\equiv2\pmod{\ell}$, this implies
		\[
		s=n-t\equiv0\pmod{\ell}.
		\]
		Since $s>0$, we may write
		\[
		s=x\ell
		\]
		for some integer $x\ge1$.
		
		Because $n\equiv2\pmod{\ell}$, the number
		\[
		D:=\frac{(\ell-1)n+2}{\ell}
		\]
		is an integer. The strict minimum-degree condition therefore
		implies
		\begin{equation}\label{dld}
			\delta(G)\ge D
			=\frac{(\ell-1)n+2}{\ell}.
		\end{equation}
		
		We now use the two external-degree estimates in turn. From
		\eqref{tlb}, \eqref{dld}, and
		$t=n-x\ell$, we obtain
		\[
		n-x\ell
		\ge\frac{(\ell-1)n+2}{\ell}-\ell+2.
		\]
		Multiplying by $\ell$ and rearranging yields
		\begin{equation}\label{xub}
			x\ell^2\le n+\ell^2-2\ell-2.
		\end{equation}
		
		Next, substitute $s=x\ell$, $t=n-x\ell$, and
		\eqref{dld} into
		\eqref{st}. This gives
		\[
		x\ell
		\left(
		\frac{(\ell-1)n+2}{\ell}-x\ell+1
		\right)
		\le(\ell-1)(n-x\ell).
		\]
		Expanding the preceding inequality gives
		\[
		x(\ell-1)n+2x+x\ell-x^2\ell^2
		\le(\ell-1)n-x\ell(\ell-1).
		\]
		Thus
		\[
		(x-1)n(\ell-1)-x^2\ell^2+x\ell^2+2x\le0,
		\]
		which is equivalent to
	   \begin{equation}\label{eq:key-external}
		(x-1)\bigl(x\ell^2-n(\ell-1)\bigr)\ge2x.
	\end{equation}
		In particular, $x\ne1$, since the left-hand side would then be
		zero while the right-hand side would be two. Hence $x\ge2$.
		Since $x-1>0$, \eqref{eq:key-external} also implies
		\begin{equation}\label{eq:x-lower}
			n(\ell-1)<x\ell^2.
		\end{equation}
		
		Combining \eqref{xub} and \eqref{eq:x-lower}, we obtain
		\[
		n(\ell-1)
		<n+\ell^2-2\ell-2,
		\]
		or equivalently,
		\begin{equation}\label{eq:final-external}
			n(\ell-2)<\ell^2-2\ell-2.
		\end{equation}
		
		If $\ell=2$, then \eqref{eq:final-external} reads
		\[
		0<-2,
		\]
		which is impossible. If $\ell\ge3$, then
		\eqref{eq:final-external} gives
		\[
		n<
		\frac{\ell^2-2\ell-2}{\ell-2}
		=\ell-\frac{2}{\ell-2}
		<\ell.
		\]
		However, $s=x\ell$ and $x\ge2$ imply
		\[
		n\ge s=x\ell\ge2\ell,
		\]
		again a contradiction.
		
		Therefore $s=0$. Hence $T$ is spanning, and Algorithm~\ref{ag}
		returns an $\ell$-congruent spanning tree of $G$.
	\end{proof}

	\section{Enumeration of $\ell$-congruent spanning trees}\label{en}
In this section, we prove Theorems~\ref{kn} and~\ref{kmn} using a
method similar to that of Feng et al.~\cite{feng2026}. Our proofs are
based on prescribed-degree enumeration formulas and a
roots-of-unity filter. We begin with two results on counting spanning
trees with prescribed vertex degrees.
	
	\begin{lemma}[\cite{berge,lovasz}]\label{kd}
		Let $d_1,d_2,\ldots,d_n$ be positive integers summing up to
		$2n-2$. Then the number of spanning trees of $K_n$ in which
		vertex $j$ has degree exactly $d_j$ for all
		$j=1,2,\ldots,n$ equals
		\[
		\frac{(n-2)!}
		{(d_1-1)!(d_2-1)!\cdots(d_n-1)!}.
		\]
	\end{lemma}
	The following result is the complete bipartite analogue of
	Lemma~\ref{kd}. It appears as Eq.~(2.2) in Moon
	\cite{moon1970}; a recent inductive proof was given by Ge and Yu
	\cite{ge2026}.
	\begin{lemma}[\cite{moon1970, ge2026}]\label{bd}
	Let $K_{m,n}$ be the complete bipartite graph with bipartition $A = \{u_1, u_2, \ldots, u_m\}$ and $B = \{v_1, v_2, \ldots, v_n\}$. Let $a_j$ and $b_h$ be positive integers such that $\sum_{j=1}^{m} a_j = \sum_{h=1}^{n} b_h = m + n - 1$. Then the number of spanning trees of $K_{m,n}$ in which the vertices $u_j$ and $v_h$ have degree exactly $a_j$ and $b_h$ for $j \in \{1, 2, \dots, m\}$ and $h \in \{1, 2, \dots, n\}$ equals
	\[
	\frac{(m-1)!(n-1)!}{\prod_{j=1}^m (a_j-1)! \prod_{h=1}^n (b_h-1)!}.
	\]
	\end{lemma}
		For a positive integer $\ell$, let $f_\ell(x)\in \mathbb{C}[[x]]$ be the formal power series defined by
		\begin{equation}
			f_{\ell}(x)=\sum_{k=0}^{\infty}\frac{x^{k\ell}}{(k\ell)!}.
		\end{equation}
The first two members of this family are \begin{equation}\label{f1}
f_1(x)=1+x+\frac{x^2}{2!}+\frac{x^3}{3!}+\cdots=\exp(x)
\end{equation} and 
\begin{equation}\label{f2}
	f_2(x)=1+\frac{x^2}{2!}+\frac{x^4}{4!}+\frac{x^6}{6!}+\cdots=\cosh(x)=\frac{\exp(x)+\exp(-x)}{2}.
	\end{equation}
Eqs.~\eqref{f1}  and \eqref{f2} are special cases of the following result. 
\begin{lemma}\label{fl}
	Let $\zeta=\exp(2\pi \iu/\ell)$. Then	\begin{equation}\label{eq:filter}
		f_\ell(x)  = \frac{1}{\ell} \sum_{j=0}^{\ell-1} \exp(\zeta^j x).
	\end{equation}
\end{lemma}
\begin{proof}
	Expanding each exponential as a formal power series gives
	\[
	\frac{1}{\ell}\sum_{j=0}^{\ell-1}\exp(\zeta^j x)
	=
	\sum_{r=0}^{\infty}
	\frac{x^r}{r!}
	\left(
	\frac{1}{\ell}
	\sum_{j=0}^{\ell-1}\zeta^{jr}
	\right).
	\]
	Since
	\[
	\frac{1}{\ell}\sum_{j=0}^{\ell-1}\zeta^{jr}
	=
	\begin{cases}
		1, & \ell\mid r,\\
		0, & \ell\nmid r,
	\end{cases}
	\]
	only the terms whose exponents are divisible by $\ell$ remain.
	Therefore,
	\[
	\frac{1}{\ell}\sum_{j=0}^{\ell-1}\exp(\zeta^j x)
	=
	\sum_{k=0}^{\infty}
	\frac{x^{k\ell}}{(k\ell)!}
	=f_\ell(x).
	\]
\end{proof}
	For $f(x)\in \mathbb{C}[[x]]$ and a nonnegative integer $m$, we use $[x^m]f(x)$ to denote the coefficient of $x^m$ in $f(x)$.
	\begin{proof}[Proof of Theorem~\ref{kn}]
		We label the vertices of $K_n$ as $v_1,\ldots,v_n$. For an
		$\ell$-congruent spanning tree $T$ of $K_n$, let
		\[
		m_j=d_T(v_j)-1,\qquad j=1,\ldots,n.
		\]
		Then each $m_j$ is a nonnegative multiple of $\ell$. Moreover,
		\[
		\sum_{j=1}^{n}m_j
		=\sum_{j=1}^{n}\bigl(d_T(v_j)-1\bigr)
		=2(n-1)-n
		=n-2.
		\]
		By Lemma~\ref{kd}, summing over all possible degree sequences gives
		\begin{equation}\label{eq:degree-sum}
			\tau_\ell(K_n)
			=
			\sum_{\substack{
					m_1+\cdots+m_n=n-2\\
					\ell\mid m_j\ \text{for }j=1,\ldots,n
			}}
			\frac{(n-2)!}{m_1!\cdots m_n!}.
		\end{equation}
		By the definition of $f_\ell(x)$, the sum in
		\eqref{eq:degree-sum} can be expressed as
		\[
		\tau_\ell(K_n)
		=(n-2)![x^{n-2}]\,f_\ell(x)^n.
		\]
		Using Lemma~\ref{fl}, we obtain
		\begin{equation}\label{tk}
			\tau_\ell(K_n)
			=
			\frac{(n-2)!}{\ell^n}[x^{n-2}]
			\left(
			\sum_{p=0}^{\ell-1}\exp(\zeta^p x)
			\right)^n,
			\qquad
			\zeta=e^{2\pi\iu/\ell}.
		\end{equation}
		By the multinomial theorem,
		\begin{align}
			\left(
			\sum_{p=0}^{\ell-1}\exp(\zeta^p x)
			\right)^n
			&=
			\sum_{\substack{
					k_0+\cdots+k_{\ell-1}=n\\
					k_0,\ldots,k_{\ell-1}\ge 0
			}}
			\frac{n!}{k_0!\cdots k_{\ell-1}!}
			\prod_{p=0}^{\ell-1}
			\left(\exp(\zeta^p x)\right)^{k_p}
			\notag\\
			&=
			\sum_{\substack{
					k_0+\cdots+k_{\ell-1}=n\\
					k_0,\ldots,k_{\ell-1}\ge 0
			}}
			\frac{n!}{k_0!\cdots k_{\ell-1}!}
			\exp\left(
			x\sum_{p=0}^{\ell-1}k_p\zeta^p
			\right).
			\label{mt}
		\end{align}
		For every $\alpha\in\mathbb{C}$, we have
		\[
		(n-2)![x^{n-2}]e^{\alpha x}
		=\alpha^{n-2}.
		\]
		Substituting \eqref{mt} into \eqref{tk} and extracting the
		coefficient of $x^{n-2}$ therefore yields
		\begin{align*}
			\tau_\ell(K_n)
			&=
			\frac{1}{\ell^n}
			\sum_{\substack{
					k_0+\cdots+k_{\ell-1}=n\\
					k_0,\ldots,k_{\ell-1}\ge 0
			}}
			\frac{n!}{k_0!\cdots k_{\ell-1}!}
			\left(
			\sum_{p=0}^{\ell-1}k_p\zeta^p
			\right)^{n-2}\\
			&=
			\frac{1}{\ell^n}
			\sum_{\substack{
					k_0+\cdots+k_{\ell-1}=n\\
					k_0,\ldots,k_{\ell-1}\ge 0
			}}
			\frac{n!}{k_0!\cdots k_{\ell-1}!}
			\left(
			\sum_{p=0}^{\ell-1}
			k_p e^{2\pi\iu p/\ell}
			\right)^{n-2}.
		\end{align*}
		This proves the theorem.
	\end{proof}	
	\begin{remark}
		We briefly demonstrate how Theorem~\ref{kn} recovers known results algebraically. For ordinary spanning trees ($\ell=1$), the degree constraint becomes vacuous, and the roots of unity collapse to $\zeta = e^{2\pi \iu } = 1$. The summation index reduces to a single term $k_0 = n$, and the formula cleanly simplifies to
		\[
		\tau_1(K_n) = \frac{1}{1^n} \frac{n!}{n!} \left( n \cdot 1 \right)^{n-2} = n^{n-2},
		\]
		recovering Cayley's formula. Similarly, for odd spanning trees ($\ell=2$), we have $\zeta = -1$. The summation reduces to combinations of $k_0$ and $k_1$ such that $k_0 + k_1 = n$. Replacing $k_0$ with $k$, we get $k_1 = n-k$, and the formula becomes
		\[
		\tau_2(K_n) = \frac{1}{2^n} \sum_{k=0}^n \binom{n}{k} \bigl(k - (n-k)\bigr)^{n-2} = \frac{1}{2^n} \sum_{k=0}^n \binom{n}{k} (2k-n)^{n-2},
		\]
		which precisely matches the formula of Feng, Chen, and Wu \cite{feng2026}.
	\end{remark}
	Next, we turn to the enumeration of $\ell$-congruent spanning trees in complete bipartite graphs. The proof of Theorem~\ref{kmn} follows a similar analytical framework, but relies on the bipartite prescribed-degree formula given in Lemma~\ref{bd}.
	\begin{proof}[Proof of Theorem~\ref{kmn}]
	Let
	\[
	A=\{u_1,\ldots,u_m\}
	\qquad\text{and}\qquad
	B=\{v_1,\ldots,v_n\}
	\]
	be the bipartition of $K_{m,n}$. For an $\ell$-congruent
	spanning tree $T$ of $K_{m,n}$, set
	\[
	a_j'=d_T(u_j)-1,\qquad j=1,\ldots,m,
	\]
	and
	\[
	b_h'=d_T(v_h)-1,\qquad h=1,\ldots,n.
	\]
	Then all $a_j'$ and $b_h'$ are nonnegative multiples of $\ell$.
	Since every edge of $T$ has one endpoint in each part and
	$|E(T)|=m+n-1$, we have
	\[
	\sum_{j=1}^{m}d_T(u_j)
	=
	\sum_{h=1}^{n}d_T(v_h)
	=
	m+n-1.
	\]
	Consequently,
	\[
	\sum_{j=1}^{m}a_j'=n-1
	\qquad\text{and}\qquad
	\sum_{h=1}^{n}b_h'=m-1.
	\]
	
	By Lemma~\ref{bd}, summing over all admissible degree sequences
	gives
	\begin{align*}
		\tau_\ell(K_{m,n})
		&=
		\sum_{\substack{
				a_1'+\cdots+a_m'=n-1\\
				\ell\mid a_j'\text{ for }j=1,\ldots,m
		}}
		\;
		\sum_{\substack{
				b_1'+\cdots+b_n'=m-1\\
				\ell\mid b_h'\text{ for }h=1,\ldots,n
		}}
		\frac{(m-1)!(n-1)!}
		{a_1'!\cdots a_m'!\,b_1'!\cdots b_n'!}\\
		&=
		(n-1)![x^{n-1}]f_\ell(x)^m\,
		(m-1)![y^{m-1}]f_\ell(y)^n.
	\end{align*}
	Using Lemma~\ref{fl}, this becomes
	\begin{align}
		\tau_\ell(K_{m,n})
		={}&
		\frac{(n-1)!}{\ell^m}[x^{n-1}]
		\left(
		\sum_{p=0}^{\ell-1}\exp(\zeta^p x)
		\right)^m
		\notag\\
		&\times
		\frac{(m-1)!}{\ell^n}[y^{m-1}]
		\left(
		\sum_{q=0}^{\ell-1}\exp(\zeta^q y)
		\right)^n,
		\qquad
		\zeta=\exp(2\pi\iu/\ell).
		\label{eq:bipartite-coeff}
	\end{align}
	
	Applying the multinomial theorem to the two powers in
	\eqref{eq:bipartite-coeff}, we obtain
	\begin{align*}
		\left(
		\sum_{p=0}^{\ell-1}\exp(\zeta^p x)
		\right)^m
		&=
		\sum_{\substack{
				k_0+\cdots+k_{\ell-1}=m\\
				k_0,\ldots,k_{\ell-1}\ge 0
		}}
		\frac{m!}{k_0!\cdots k_{\ell-1}!}
		\exp\left(
		x\sum_{p=0}^{\ell-1}k_p\zeta^p
		\right),
	\end{align*}
	and, similarly,
	\begin{align*}
		\left(
		\sum_{q=0}^{\ell-1}\exp(\zeta^q y)
		\right)^n
		&=
		\sum_{\substack{
				s_0+\cdots+s_{\ell-1}=n\\
				s_0,\ldots,s_{\ell-1}\ge 0
		}}
		\frac{n!}{s_0!\cdots s_{\ell-1}!}
		\exp\left(
		y\sum_{q=0}^{\ell-1}s_q\zeta^q
		\right).
	\end{align*}
	Since
	\[
	r![z^r]\exp(\alpha z)=\alpha^r,
	\]
	extracting the required coefficients in
	\eqref{eq:bipartite-coeff} yields
	\begin{align*}
		\tau_\ell(K_{m,n})
		={}&
		\frac{1}{\ell^{m+n}}
		\left[
		\sum_{\substack{
				k_0+\cdots+k_{\ell-1}=m\\
				k_0,\ldots,k_{\ell-1}\ge 0
		}}
		\frac{m!}{k_0!\cdots k_{\ell-1}!}
		\left(
		\sum_{p=0}^{\ell-1}k_p\zeta^p
		\right)^{n-1}
		\right]
		\\
		&\quad\times
		\left[
		\sum_{\substack{
				s_0+\cdots+s_{\ell-1}=n\\
				s_0,\ldots,s_{\ell-1}\ge 0
		}}
		\frac{n!}{s_0!\cdots s_{\ell-1}!}
		\left(
		\sum_{q=0}^{\ell-1}s_q\zeta^q
		\right)^{m-1}
		\right].
	\end{align*}
	Substituting $\zeta=\exp(2\pi\iu/\ell)$ proves the theorem.
\end{proof}

	\begin{remark}
		By a similar algebraic reduction, Theorem~\ref{kmn} yields the corresponding formulas for complete bipartite graphs. Setting $\ell=1$ gives $\zeta = 1$, and the sums collapse to single terms $k_0 = m$ and $s_0 = n$, yielding
		\[
		\tau_1(K_{m,n}) = \frac{1}{1^{m+n}} \left[ \frac{m!}{m!} m^{n-1} \right] \left[ \frac{n!}{n!} n^{m-1} \right] = m^{n-1} n^{m-1}.
		\]
		Setting $\ell=2$ and $\zeta=-1$, the sums resolve into two independent binomial expansions, giving
		\[
	\tau_2(K_{m,n}) = \frac{1}{2^{m+n}} \left[ \sum_{k=0}^m \binom{m}{k} (2k-m)^{n-1} \right] \left[ \sum_{s=0}^n \binom{n}{s} (2s-n)^{m-1} \right],
	\]
		which recovers the complete bipartite odd-spanning-tree formula of Ge and Yu \cite{ge2026} and Xu and Xu \cite{xu2026}.
	\end{remark}
\section*{Declaration of competing interest}
	The authors declare that they have no known competing financial interests or personal relationships that could have
	appeared to influence the work reported in this paper.
		\section*{Acknowledgments}
	This work is partially supported by the National Natural Science Foundation of China (Grant No. 12001006) and Wuhu Science and Technology Project, China (Grant No. 2024kj015). The authors acknowledge the use of Gemini 3.1 Pro for language polishing and detailed proofreading of the manuscript.
	
\end{document}